\documentclass[12pt]{article}

\usepackage{amssymb,amsmath}

\newcommand{\R}{\mathbb{R}}

\newcommand{\dom}{{\rm dom\,}}
\newcommand{\Dom}{{\rm Dom\,}}

\newcommand{\Gr}{{\rm Gr}}

\newcommand{\beq}{\begin{equation}}
\newcommand{\eeq}{\end{equation}}

\newcommand{\eps}{{\varepsilon}}

\newcommand{\tos}{\rightrightarrows}

\newenvironment{proof}
         {\begin{trivlist}\item[
         {\bf Proof.}]}{{\hfill $\square$} \end{trivlist}}

\newtheorem{theo}{Theorem}[section]

\newtheorem{rem}[theo]{Remark}
\newtheorem{lemma}[theo]{Lemma}
\newtheorem{prop}[theo]{Proposition}

\newtheorem{cor}[theo]{Corollary}

\begin{document}

\title{\LARGE\bf Generic variational principles for supinf problems with constraints}
\author{\bf  D. Gaumont, D. Kamburova and J.P. Revalski\thanks{The work of the third author was partially supported by the Centre of Excellence in Informatics and ICT under the Grant No BG16RFPR002-1.014-0018-C01, financed by the Research, Innovation and Digitalization for Smart Transformation Programme 2021-2027 and co-financed by the European Union}}

\date{}

\maketitle

\begin{abstract}

We study supinf problems for perturbations of a given function of two variables by continuous bounded functions in the underlying spaces and the properties of the associated solution mapping. We provide conditions on the initial data such that most (in the Baire category sense) of the functions from the perturbation space  secure existence and uniqueness of the solution of the supinf problem for the corresponding perturbed functions. In this way we obtain generic variational principles for such problems. 

\end{abstract}



\section{Introduction and preliminary results}
Let $X$ and $Y$ be completely regular topological spaces and $f: X \times Y \rightarrow [-\infty, \infty]$ be an extended real-valued function (which we will suppose to be finite at the points of some nonempty set of the Cartesian product $X\times Y$). We will be interested in the following (constrained) supinf problem:
$$
(P) \qquad \sup_{x \in X} \inf_{y \in Kx} f(x,y),
$$
where the  set-valued mapping  $K: X\rightrightarrows Y$, giving the constraints,  is with nonempty images.  A solution to the problem (P) is a couple  $(x_0,y_0)\in X\times Y$ which satisfies the following conditions:

$$
f(x_0,y_0)=\inf_{y\in Kx_0} f(x_0,y)=\sup_{x\in X}\inf_{y\in Kx} f(x,y).
$$
The quantity 
$$
v_f:=\sup_{x\in X}\inf_{y\in Kx} f(x,y)
$$
is called {\it value of the problem} (P).

\smallskip

In such a setting fall, for example, the so-called  "leader-follower" games, in which the first player (the leader) makes her/his choice first, with the aim to maximize her/his profit which is given by the function $f$. When the choice of the first player $x\in X$ is made, the second player makes her/his choice in its set of feasible choices $Kx$. The guaranteed maximal utility which the first player can secure, irrespective from the behaviour of the second one, is the value of the corresponding supinf problem. When the set-valued mapping $K$ is given as the set of solutions to some optimization problem, i.e. for any $x\in X$ we have
$$
Kx:=\{y': g(x,y')=\inf_{y\in Y} g(x,y)\},
$$
for some given function $g:X\times Y\to \R$, this game is known as Stackelberg problem (called also "two level optimization problem") - see e.g \cite{H_Stack}, \cite{LoM1}.

In the unconstrained case, i.e., when we have $Kx=Y$ for any $x\in X$, the value of the problem (P) gives, for example, the guaranteed gain for some of the two players in noncooperative games. In the unconstrained case, the problem (P) is also related to the question of existence of saddle points for the function $f$.

\smallskip

Our aim in this article is to study perturbations of the given function $f$ by continuous bounded functions  defined on $X\times Y$ (the latter, considered with some usual product topology) which secure (at least) existence of the solution of the problem (P) for the perturbed function. In particular, we will be interested in such perturbations consisting of additively separable continuous bounded functions defined on $X$ and $Y$, respectively. We investigate some important properties of the solution set-valued mapping which assigns to a given perturbation the solutions to the corresponding perturbed supinf problems. This allows us to provide conditions on the data which entail   that the set of continuous bounded perturbations of $f$ which provide  existence and uniqueness (even stronger, {\it well-posedness} - see the definition in Section 3) of the supinf problem for the perturbed function form a dense and $G_\delta$-subset of the corresponding Banach space of perturbations. Thus, most of the perturbations (in the Baire category sense) give rise well-posed perturbed problems and this constitutes {\it a generic variational principle}  for such problems.

Let us mention that such generic variational principles have been extensively studied in the case of one-level optimization - see, e.g,  the following, far to be exhaustive, list of references \cite{CKR1,CKR3, DGoZ1, DGoZ2, ILR, IZa, IvKR,  LaR,TZl2,S}. On the contrary, the list of such principles for supinf problems, to the best knowledge of the authors,  is rather short - see e.g. \cite{KL,KRi}, where generic variational principles are proved, when the objective function is the zero one, and \cite{GKaR,KR2}, where  dense variational principles are proved for the setting above. 

To introduce some preliminary notions and results, let us first recall  the two most classical notions of continuity of a set-valued mapping  $\Phi:Z\tos T$ between two topological spaces $Z$ and $T$. $\Phi$ is called {\it upper semicontinuous} (in short usc)  at $z_0$ if for every nonempty open set $V$ of $T$ containing $\Phi(z_0)$ there is a nonempty open set $U$ of $Z$ containing $z_0$ such that $\Phi(z)\subset V$ for every $z\in U$. The mapping $\Phi$ is usc in $Z$ if it is usc at any point of $Z$. When $\Phi$ is with nonempty images, it is called {\it lower semicontinuous} (in short lsc) at $z_0$ if  for every nonempty open set $V$ of $T$ intersecting  $\Phi(z_0)$ there is a nonempty open set $U$ of $Z$ containing $z_0$ such that $\Phi(z)\cap  V\neq\emptyset$ for every $z\in U$ and  $\Phi$ is lsc in $Z$ if it is lsc at any point of $Z$. Since we will come across with set-valued mappings which can have empty images at some points, let us recall that for a given set-valued mapping $\Phi:Z\tos T$, the set $\Dom(\Phi):=\{z\in Z: \Phi(z)\neq\emptyset\}$  denotes the {\it domain} of $\Phi$. As usual, for such a mapping, the set $\Gr(\Phi):=\{(z,t): t\in \Phi(z)\}$ is the {\it graph} of $\Phi$. 

\smallskip 

 It has turned out in preceding considerations of supinf problems  that, in order to expect reasonable properties of such  problems to hold, the following conditions on the initial data, the objective function  $f:X\times Y\to [-\infty,+\infty]$ and the set-valued mapping $K:X\tos Y$, seem natural:
 
 \begin{itemize}

 \item[{\rm (A1})] the function $w(\cdot)=\inf_{y\in K(\cdot)} f(\cdot,y)$ is bounded from above on $X$ and proper as a function with values in $\R\cup\{-\infty\}$;

 \item[{\rm (A2})]  $f$ is upper semicontinuous in $X \times Y$;

 \item[{\rm (A3)}] for any $x\in X$ the function $f(x,\cdot)$ is lower semicontinuous in $Y$.

\end{itemize}

Here the term {\it proper} for a given extended real-valued function   means, as usual, that the function has at least one finite value. In particular, condition (A1) implies that the objective function $f$ must be also proper, that is, the {\it effective domain} $\dom(f):=\{ (x,y)\in X\times Y: -\infty<f(x,y)<\infty\}$ of $f$ is not empty.

\smallskip

For a topological space $Z$, denote by $C(Z)$ the family of all bounded and continuous real-valued functions in $Z$. Equipped with the usual sup-norm $\Vert g\Vert_{Z,\infty}:=\sup\{\vert g(z)\vert: z\in Z\}$, $g\in C(Z)$, the space $(C(Z),\Vert\cdot\Vert_{Z,\infty})$ is a real Banach space. The following perturbation lemma (where, as usual $\R_+$ designates the set of nonnegative real numbers)  was first proved in \cite{KR2} for the unconstrained case (and in this case only upper semicontinuity of $f$ with respect to the first variable is enough) and then in \cite{GKaR} in the constrained setting:

\begin{lemma} \label{basic_lemma}
Let $X$ and $Y$ be completely regular topological spaces, $f:X\times Y\to [-\infty,+\infty]$ be an extended-real valued function and $K:X\tos Y$ be a nonempty-valued set-valued mapping which satisfy conditions (A1)-(A3).
Suppose that $K: X \rightrightarrows Y$ is lsc with closed images. Let $\varepsilon>0$ and $x_0 \in X$ be such that $w(x_0)>sup_{x \in X} w(x) - \varepsilon$ and let $\delta>0$ and $y_0 \in Kx_0$ be such that $f(x_0,y_0)<\inf_{y \in Kx_0} f(x_0,y)+\delta$. Then, there exist continuous bounded functions $g_0: X \rightarrow \R_+$ and $h_0: Y\to \R_+$, such that $g_0(x_0)=h_0(y_0)=0$, $||g_0||_{X,\infty} \leq \varepsilon$, $||h_0||_{Y,\infty} \leq \delta$ and the supinf problem $\sup_{x \in X} \inf_{y \in Kx} \{f(x,y)-g_0(x)+h_0(y)\}$ has a solution at $(x_0,y_0)$.

\end{lemma}

Notice that, under the conditions of the lemma, points $x_0,y_0$ as in the statement of the lemma always exist.

The rest of the paper is organized as follows: in the next Section 2 we study several important  properties of solution  mappings associated with supinf problems for the perturbed functions. In Section 3 we provide conditions on the initial data in order to have generic variational principles (of the above mentioned type) for supinf problems.

\section{Properties of solution mappings for supinf problems}

In this section,  given  completely regular topological spaces $X$ and $Y$,  an extended real-valued function $f:X\times Y\to [-\infty,+\infty]$ and a set-valued mapping $K:X\tos Y$ with nonempty images, we will investigate the properties of solution mappings which assign the solutions of the perturbed supinf problem, when the perturbations are continuous bounded functions in $X\times Y$.

The first one is when we perturb $f$  by a sum of functions from $C(X)$ and $C(Y)$, namely, 
$$
S_f: C(X) \times C(Y) \rightrightarrows X \times Y,
$$
which assigns to every couple of functions $(g,h) \in C(X)\times  C(Y)$ the set of solutions (possibly empty) of the supinf problem 
$$
\sup_{x \in X} \inf_{y \in Kx} \{f(x,y)+g(x)+h(y)\},
$$
with constraints given by the set-valued mapping $K:X\tos Y$. In $C(X)$ and $C(Y)$ we consider the usual sup norms $\Vert g\Vert_{X,\infty}$ and $\Vert h\Vert _{Y,\infty}$ correspondingly, and on the Cartesian product $C(X)\times C(Y)$ some usual product norm generated by the norms in $C(X)$ and $C(Y)$ respectively.

For our subsequent considerations we will need sometimes to impose  stronger conditions on the continuity properties of the objective function than those in (A2)-(A3). More precisely, we will consider the following condition:

 \begin{itemize}

 \item[{\rm (A2$'$})]  The function $f:X\times Y\to [-\infty,\infty]$ is continuous.

 \end{itemize}
 
Here the continuity of $f$ is understood in the usual extended sense. Evidently, (A2$'$) implies (A2) and (A3). Condition (A2$'$) implies also that the domain of $f$ is open in $X\times Y$. 

The following result is the well-known Berge theorem, here given for an extended real-valued function. Its proof follows the path of the proof of the classical Berge theorem, taking into consideration that the objective function (and the infimal value function) may have points where the values are not finite.

\begin{theo}(Berge)\label{Berge}
Let $X$ and $Y$ be completely regular topological spaces,  and a given   extended real-valued function   $f:X\times Y\to [-\infty, +\infty]$ and a set-valued mapping $K:X\tos Y$ with nonempty images satisfy the conditions (A1)-(A2\,$'$). Suppose, in addition,  that  $K$ is both lsc and usc in $X$ and with compact images. Then, the function $\inf_{y \in K(\cdot)} f(\cdot,y):X\to [-\infty,+\infty)$ is proper and continuous. 
\end{theo}

The following remark is evident and will be used frequently in the sequel.

\begin{rem}\label{basic_remark}
{\rm
If a function $f:X\times Y\to[-\infty,\infty]$ satisfies the conditions (A1)-(A3) (resp. (A1)-(A2$'$)) with respect to the mapping $K:X\tos Y$, then for any continuous and bounded real-valued function $u:X\times Y\to  \R$, the function $f+u$ also satisfies the conditions (A1)-(A3) (resp. (A1)-(A2$'$)) with respect to the mapping $K$. Moreover,  for such an $u$, the domain $\dom(f)$ of $f$ is the same as the domain of $f+u$, and the domain   $\dom(w)$ of the function $w$ from the assumption (A1) is the same as the domain of the function $\inf_{y\in K(\cdot)}\{f(\cdot,y)+u(\cdot,y)\}$. Thus, Lemma \ref{basic_lemma} and Berge Theorem \ref{Berge} hold also for any such  function $f+u$.  In particular, all these properties are true for any additively separable function $u:=g+h$, with $g\in C(X)$ and $h\in C(Y)$.
 
}
\end{rem}

Before proving the next result, let us recall a notion of continuity for a set-valued mapping, which generalizes the similar one for single-valued mappings given by Kempisty \cite{Ke}:  a set-valued mapping $\Phi: Z\tos T$ between the topological spaces $Z$ and $T$ is called {\it quasi-continuous}  if for every two open sets $U$ in $Z$ and $V$ in $T$ such that $\Phi(U)\cap V\neq\emptyset$, there exists a nonempty open set $U'$ in $Z$ such that $U'\subset U$ and $\Phi(U')\subset V$.  When the mapping $\Phi$ is usc and with nonempty compact values in $Z$, this property characterizes the fact that $\Phi$ is a {\it minimal } mapping in the sense of graph inclusions among the usc mappings with nonempty compact values between $Z$ and $T$- see Christensen \cite{Chr}.

The following proposition gives some of the fundamental  properties of the above introduced solution mapping. As usual, given a metric space $(Z,d)$, the symbol $B_Z(z,r)$ designates the open ball with center at $z\in Z$ and radius $r>0$. If $Z$ is a topological space and $A$ is its subset, the symbol $\overline{A}$ means the closure of $A$ in $Z$.  

\begin{prop}\label{solution_map_1}
Let $X$ and $Y$ be completely regular topological spaces,  and a given   extended real-valued function   $f:X\times Y\to [-\infty, +\infty]$ and a set-valued mapping $K:X\tos Y$ with nonempty images satisfy the conditions (A1)-(A2\,$'$). Suppose, in addition, that  $K$ is both lsc and usc in $X$ and with compact images. Then, the solution mapping $S_f:C(X)\times C(Y)\tos X\times Y$  possesses the following properties:

 \begin{enumerate}
 
 \item[{\rm (a)}] the domain $\Dom(S_f)$ is dense in $C(X)\times C(Y)$;
 
 \item[{\rm (b)}] the graph $\Gr(S_f)$ of $S_f$ is closed in the product topology of $C(X)\times C(Y)$ and $X\times Y$;

 \item[{\rm (c)}] the mapping $S_f$ is quasi-continuous;
    
\item[{\rm (d)}] the mapping $S_f$ is open as a mapping from $C(X)\times C(Y)$  into $\Gr(K)$.
 
 \end{enumerate}
 
\end{prop}

\begin{proof}

As mentioned in Theorem 2 from \cite{GKaR} condition (a) is a direct consequence of Lemma \ref{basic_lemma},  having in mind also Remark \ref{basic_remark} (and this is true only by assuming the assumptions (A1)-(A3)).  

\medskip

To prove (b), suppose that we have nets $\{(g_\alpha,h_\alpha)\}_{\alpha \in \Lambda}\subset C(X)\times C(Y)$ and $\{(x_\alpha,y_\alpha)\}_{\alpha\in \Lambda}\subset X\times Y$ such that $\{(g_\alpha,h_\alpha)\}_{\alpha \in \Lambda}$ converges (uniformly) to $(g_0,h_0)\in C(X)\times C(Y)$, $\{(x_\alpha,y_\alpha)\}_{\alpha\in \Lambda}$ converges to $(x_0,y_0)\in X\times Y$ and we have $(x_\alpha,y_\alpha)\in S_f(g_\alpha,h_\alpha)$ for every $\alpha \in \Lambda$. The latter means that, for every $\alpha \in \Lambda$,  $y_\alpha\in Kx_\alpha$ and, moreover 
$$
\begin{array}{ll} \vspace{6pt}
f(x_\alpha,y_\alpha)+g_\alpha(x_\alpha)+h_\alpha(y_\alpha)&=\inf_{y\in Kx_\alpha}\{f(x_\alpha,y)+g_\alpha(x_\alpha)+h_\alpha(y)\} \\ \vspace{6pt}
&=\sup_{x\in X}\inf_{y\in Kx}\{f(x,y)+g_\alpha(x)+h_\alpha(y)\} \\ \vspace{6pt}
&=v_{f+g_\alpha+h_\alpha}.
\end{array}
\leqno{(1)}
$$
\noindent 
We have to show that $(x_0,y_0)\in S_f(g_0,h_0)$. 

First, since $K$ is usc with nonempty compact values, it has a closed graph in $X\times Y$, and consequently, $y_0\in Kx_0$. Further,
observe that condition (A1) implies that for any function $f+g+h$ where $g\in C(X)$ and $h\in C(Y)$, when determining the value of the problem $v_{f+g+h} =\sup_{x\in X} \inf_{y\in Kx}\{f(x,y)+g(x)+h(y)\}$, the supremum can be taken only on $x$ in the domain of the function $\inf_{y\in K(\cdot)}\{f(\cdot,y)+g(\cdot)+h(y)\}$ and the infimum only on those $y\in Kx$ for which the value $f(x,y)$ is finite (such an $y\in Kx$ always exists according to (A1) and the choice of $x$). That is, the value of the supinf problem for $f+g+h$ depends only on finite values of $f$. 

Having in mind the last observation, it is easily seen that,  because of the uniform convergence of  $(g_\alpha,h_\alpha)$  to $(g_0,h_0)$,  the value  $v_{f+g_\alpha+h_\alpha}$ of the supinf problem for the function $f+g_\alpha+h_\alpha$ converges with $\alpha$ to the value $v_{f+g_0+h_0}$ of the supinf problem for the function $f+g_0+h_0$. Since, according to (A1), the latter value is finite, having in mind (1), and again the uniform convergence of $(g_\alpha,h_\alpha)$ to $(g_0,h_0)$ and the convergence of $(x_\alpha,y_\alpha)$ to $(x_0,y_0)$,  this implies that the value $f(x_0,y_0)$ is finite and that 

$$
\begin{array}{ll} \vspace{6pt}
f(x_0,y_0)+g_0(x_0)+h_0(y_0)&=\lim_\alpha\{f(x_\alpha,y_\alpha)+ g_\alpha(x_\alpha)+h_\alpha(y_\alpha)\}\\ \vspace{6pt}
&=\lim_\alpha v_{f+g_\alpha+h_\alpha}=v_{f+g_0+h_0}.
\end{array}
\leqno{(2)}
$$
Moreover, again using the uninform convergence of $(g_\alpha,h_\alpha)$ to $(g_0,h_0)$, the functions $\inf_{y\in K(\cdot)}\{f(\cdot,y)+ g_\alpha(\cdot) +h_\alpha(y)\}$ (which are continuous according to Berge Theorem \ref{Berge}) uniformly converges (on their  common domain) with $\alpha$ to the continuous function $\inf_{y\in K(\cdot)}\{f(\cdot,y)+ g_0(\cdot) +h_0(y)\}$. Using this, the convergence of $(x_\alpha,y_\alpha)$ to $(x_0,y_0)$, and having in mind also (1)  and (2), we have
 $$
\begin{array}{ll} \vspace{6pt}
f(x_0,y_0)+g_0(x_0)+h_0(y_0)&=\lim_\alpha\{f(x_\alpha,y_\alpha)+ g_\alpha(x_\alpha)+h_\alpha(y_\alpha)\}\\ \vspace{6pt}
&=\lim_\alpha \inf_{y\in K x_\alpha} \{f(x_\alpha,y)+g_\alpha(x_\alpha)+h_\alpha(y)\} \\
& = \inf_{y\in Kx_0}\{f(x_0,y)+g_0(x_0)+h_0(y)\},
\end{array}
$$
which shows, together with  (2), that $(x_0,y_0)$ is a solution of the supinf problem for the function $f+g_0+h_0$, and this completes the proof that the solution mapping $S_f$ has closed graph.

\medskip 

Let us now  prove (c). Take a nonempty open set $W$ of $C(X)\times C(Y)$, and open sets $U$ of $X$ and $V$ of $Y$, such that  $S_f(W)\cap (U\times V)\neq\emptyset$. We have to find a nonempty open set $W'$ of $C(X)\times C(Y)$ such that $W'\subset W$,  and moreover, $S_f(W')\subset U\times V$. To this end, let $(g_0,h_0)\in W$ and $(x_0,y_0)\in X\times Y$ be such that $(x_0,y_0)\in S_f(g_0,h_0)\cap (U\times V)$.
In order  to shorten the length of the expressions below, we can use  Remark \ref{basic_remark} and the simple observation that $S_f(g_1+g_2, h_1+ h_2)=S_{f+g_1+h_1}(g_2,h_2)$ for any $g_1,g_2\in C(X)$ and $h_1,h_2\in C(Y)$,   to suppose, without loss of generality, that $g_0\equiv 0$ and $h_0\equiv 0$. In this case $f(x_0,y_0)=\inf_{y\in Kx_0} f(x_0,y)=\sup_{x\in X} \inf_{y\in Kx} f(x,y)$.

Let $\eps>0$ be such that
$$
B_{C(X)}(\theta_X,8\eps)\times B_{C(Y)}(\theta_Y, 8\eps)\subset W,
\leqno{(3)}
$$
where $\theta_X$ and $\theta_Y$ are the origins in $C(X)$ and $C(Y)$ respectively. 

Let further $U_1'$ and $V_1'$ be nonempty open sets in $X$ and $Y$ correspondingly,  such that $x_0\in U_1'\subset \overline{U}_1'\subset U$,  $y_0\in V_1'\subset \overline{V}_1'\subset V$ and which satisfy, in addition, the following conditions:
 $$
 \vert f(x,y)-f(x_0,y_0)\vert <\eps/10 \mbox{ for every } (x,y)\in U_1'\times V_1',
 \leqno{(4)}
 $$
 and 
 $$
 \vert \inf_{y\in Kx} f(x,y)-\inf_{y\in Kx_0} f(x_0,y)\vert <\eps/10 \mbox{ for every } x\in U_1'.
 \leqno{(5)}
 $$
The property in (4) is possible because of the continuity of $f$ at $(x_0,y_0)$ and the fact that $f(x_0,y_0)$ is finite, while  (5) follows by  Berge theorem  (Theorem \ref{Berge}), again taking into account that $f(x_0,y_0)=\inf_{y\in Kx_0} f(x_0,y)$ is finite. 
  
Let, further, $q_0\in C(Y)$ be such that $0\le q_0(y)\le 1$, for any $y\in Y$,  $q_0(y_0)=0$ and $q_0|_{Y\setminus V_1'}\equiv 1$. Consider the open (in $Y$) set  $V_1'':=\{y\in V_1': q_0(y)<1/2\}$. Since $y_0\in K{x_0}\cap V_1''$ and $K$ is lsc, there exists  a nonempty open set $U_1$ in $X$ such that $x_0\in U_1\subset \overline{U}_1\subset U_1'$ for which  
$$
Kx\cap V_1''\neq \emptyset \mbox{ for every } x\in \overline{U}_1.
\leqno{(6)}
$$
Let $p_0\in C (X)$ be such that $0\le p_0(x)\le 1$ for $x\in X$, $p_0(x_0)=0$ and $p_0|_{X\setminus U_1}\equiv 1$. 
Set $V_1:=V_1'$. 

Let us consider the couple  $(-7\eps p_0, \eps q_0)\in C(X)\times C(Y)$. Then, using also (3), it is seen that the nonempty open (in $C(X)\times C(Y)$) set $W'$ defined below satisfies
$$
W':=B_{C(X)}(-7\eps p_0,\eps/10)\times B_{C(Y)}(\eps q_0,\eps/10)\subset W.
$$
We will show that $S_f(W')\subset \overline{U}_1\times \overline{V}_1$ and this will complete the proof of the assertion because $\overline{U}_1\times \overline{V}_1 \subset U\times V$ . To see this, let $(g,h)\in W'$ and let $(\bar x,\bar y)\in S_f(g,h)$. Then,  $\bar y\in K\bar x$ and we have 
$$
\begin{array}{ll}  \vspace{6pt}
f(\bar x, \bar y)+g(\bar x)+h(\bar y) & = \inf_{y\in K\bar x}\{f(\bar x,y)+g(\bar x)+h(y)\} \\  \vspace{6pt}
&  =\sup_{x\in X} \inf_{y\in Kx} \{f(x,y)+g(x)+h(y)\}.
\end{array}
\leqno{(7)}
$$
We know also that
$$
f(x_0,y_0)=\inf_{y\in Kx_0} f(x_0,y)=\sup_{x\in X} \inf_{y\in Kx} f(x,y)
\leqno{(8)}
$$
because initially $(x_0,y_0)\in S_f(\theta_X,\theta_Y)$ (in particular, we have also $y_0\in Kx_0)$.

Recall that 
$$
w(x):=\inf_{y\in Kx} f(x,y), \, \, x\in X,
$$
and set 
$$
w'(x):=\inf_{y\in Kx} \{f(x,y)+g(x)+h(y)\}, \, \, x\in X.
$$

We may also write that $g=-7\eps p_0+g'$ and $h= \eps q_0+h'$ with some $g'\in C(X)$, $h'\in C(Y)$ such that $\Vert g'\Vert_{X,\infty}<\eps/10$ and $\Vert h'\Vert_{Y,\infty}<\eps/10$.

Suppose first that $\bar x\notin \overline{U}_1$. Then, $p_0(\bar x)=1$ and taking into account also the second equality in (8), the fact that $p_0(x_0)=0$  and the choices of the functions above, we have
$$
\begin{array}{ll}  \vspace{6pt}
w'(\bar x) & =\inf_{y\in K\bar x} \{f(\bar x,y)+g(\bar x) +h(y)\} \\  \vspace{6pt}
& =\inf_{y\in K\bar x}\{ f(\bar x,y) -7\eps p_0(\bar x)+ g'(\bar x) +\eps q_0(y)+ h'(y)\}\\  \vspace{6pt}
& \le \inf_{y\in K\bar x} \{ f(\bar x,y)- 4 \eps\}=\inf_{y\in K\bar x} f(\bar x,y)-4\eps \\  \vspace{6pt}
& = w(\bar x)-4\eps\le  w(x_0)-4 \eps \\  \vspace{6pt}
& =\inf_{y\in Kx_0} f(x_0,y)-4\eps\le \inf_{y\in Kx_0} \{ f(x_0,y)+g(x_0)\}-3\eps \\  \vspace{6pt}
& < \inf_{y\in Kx_0} \{ f(x_0,y)+ g(x_0)+h(y)\}-\eps \\  \vspace{6pt}
&=w'(x_0)-\eps <w'(x_0). 
\end{array}
$$
Therefore, $w'(\bar x)<w'(x_0)$ and this is a contradiction, since $w'(\bar x)=\sup_{x\in X} w'(x)$. Thus, $\bar x\in \overline{U}_1$.

Further, we show that $\bar y\in \overline{V}_1$. Suppose, on the contrary, that $\bar y\notin \overline{V}_1$. Then $q_0(\bar y)=1$. Since $\bar x\in \overline{U}_1$ by (6) there is $\tilde y\in K\bar x\cap V_1''\subset K \bar x\cap V_1$. Then, having in  mind also (7), (4), (5)  and the choice of $q_0$,  we have
$$
\begin{array}{ll} \vspace{6pt}
\inf_{y\in K\bar x} \{ f(\bar x,y)+ g(\bar x)+ h(y)\} &=f(\bar x,\bar y)+g(\bar x)+ h(\bar y )\\  \vspace{6pt}
& \ge \inf_{y\in K \bar x} f(\bar x,y) +g(\bar x) +\eps q_0(\bar y)+ h'(\bar y) \\  \vspace{6pt}
&\ge \inf_{y\in Kx_0} f(x_0,y)-\eps/10 +g(\bar x)+\eps-\eps/10 \\ \vspace{6pt}
& =f(x_0,y_0) +g(\bar x)+\eps-2\eps/10 \\ \vspace{6pt}
& \ge f(\bar x, \tilde y)+g(\bar x)+\eps-3\eps/10 \\ \vspace{6pt}
& > f(\bar x, \tilde y)+g(\bar x) +\eps q_0(\tilde y)-\eps/2 +\eps-3\eps/10 \\ \vspace{6pt}
& \ge f(\bar x, \tilde y)+g(\bar x) +h(\tilde y)-\eps/10 +\eps-8\eps/10 \\ \vspace{6pt}
& = f(\bar x, \tilde y)+g(\bar x) +h(\tilde y)+\eps/10 \\ \vspace{6pt}
& >  f(\bar x, \tilde y)+g(\bar x) +h(\tilde y)
\end{array}
$$ 
and this is a contradiction since $\tilde y\in K\bar x$. Therefore, $\bar y\in \overline{V}_1$, which entails that $S_f(g,h)\in \overline{U}_1\times \overline{V}_1\subset U\times V$ and this completes the proof that the mapping $S_f$ is quasi-continuous.

\medskip
It remains to prove (d). Let $O$ be a nonempty open subset of $C(X)\times C(Y)$. We know from the assertion (a)  above that $S_f(O)$ is nonempty. Take an arbitrary point $(x_0,y_0)\in S_f(O)$. Then for some $(g_0,h_0)\in O$ we have $(x_0,y_0)\in S_f(g_0,h_0)$ which means that 
$$
\begin{array}{ll} \vspace{6pt}
f(x_0,y_0)+g_0(x_0)+h_0(y_0)&=\inf_{y\in Kx_0}\{f(x_0,y)+g_0(x_0)+h_0(y)\} \\  \vspace{6pt}
 & = \sup_{x\in X} \inf_{y\in Kx}\{f(x,y)+g_0(x)+h_0(y)\}.
 \end{array}
 \leqno{(9)}
$$
Since $O$ is open in $C(X)\times C(Y)$, there is $\eps>0$ such that 
$$
B_{C(X)}(g_0, 2\eps)\times B_{C(Y)}(h_0,2\eps)\subset O.
$$
 Let $U$ and $V$ be open subsets of  $X$ and $Y$ respectively, such that  $x_0\in U$, $y_0\in V$ and we have, in addition, 
$$
\vert f(x,y)+g_0(x)+h_0(y)-f(x_0,y_0)-g_0(x_0)-h_0(y_0)|<\eps \,   \forall  (x,y)\in U\times V,
\leqno{(10)}
$$
and
$$
\vert \inf_{y\in Kx} \{f(x,y)+g_0(x)+h_0(y)\}- \inf_{y\in Kx_0}\{f(x_0,y)+g_0(x_0)+h_0(y)\} \vert < \eps \, \forall x\in U.
\leqno{(11)}
$$
The property in (10) is possible because of the continuity of $f$, $g$ and $h$, and the fact that $f$ is finite at $(x_0,y_0)$, and the property in (11) because of Berge theorem and the fact that $\inf_{y\in K(\cdot)}\{f(\cdot,y)+g_0(\cdot)+h_0(y)\}$ is also finite at $x_0$. We may suppose  also that $Kx\cap V\neq\emptyset$ for every $x\in U$, because of the lower semicontinuity of $K$.

Take now an arbitrary couple $(\bar x,\bar y)\in U\times V$, such that $\bar y\in K\bar x$. Then, by (11) we have 
$$
\inf_{y\in K\bar x} \{f(\bar x,y)+g_0(x)+h_0(y)\}> \inf_{y\in Kx_0}\{f(x_0,y)+g_0(x_0)+h_0(y)\}-\eps
$$
which together with the second equality in (9) show that the point $\bar x$ satisfies the first condition of Lemma \ref{basic_lemma} for the objective function $f+g_0+h_0$. On the other hand, by (10) and the first equality in (9) we deduce that
$$
\begin{array}{ll} \vspace{6pt}
f(\bar x,\bar y)+g_0(\bar x)+h_0(\bar y) & < f(x_0,y_0)+g_0(x_0)+h_0(y_0)+\eps \\ \vspace{6pt}
&= \inf_{y\in Kx_0}\{f(x_0,y)+g_0(x_0)+h_0(y)\} +\eps,
\end{array}
$$
which shows that $\bar  y$ satisfies the second condition in Lemma \ref{basic_lemma} with $\delta=\eps$, again for the function $f+g_0+h_0$. Thus, having in mind Remark \ref{basic_remark}, there exists $(\bar g,\bar h)\in C(X)\times C(Y)$ with  $\Vert \bar g\Vert_{X,\infty}  \le \eps$,  $\Vert \bar h\Vert_{Y,\infty}  \le \eps$, such that $(\bar x,\bar y)$ is a solution for the supinf problem with objective function $f+g_0+h_0-\bar g+\bar h$. Because of the initial choice of $\eps$ we obviously have $(g_0-\bar g,h_0+ \bar h)\in O$ and therefore,  $(\bar x,\bar y)\in S_f(O)$. Since $(\bar x,\bar y)$, with $\bar y\in K\bar x$, was arbitrary taken from $U\times V$, the latter shows that $(U\times V)\cap \Gr(K)\subset S_f(O)$. This completes the proof of $(d)$ and of the proposition.

\end{proof}

When making perturbations of a given function, in order to obtain desirable good properties for the perturbed function (like above, existence of a solution) the general idea is to deal with a  family of perturbations which are as simple as possible.  From that point of view our interest so far have been concentrated on perturbations of a function of two variables by  additively separable continuous bounded functions. The latter  are easier to handle than merely a continuous bounded function of two variables. But when the question is also to study the structure of the set of good perturbations, it is also of interest to study perturbation spaces which are larger. More precisely, one can address from this point of view also the question of perturbations of a given function $f:X\times Y\to [-\infty,+\infty]$, which satisfies, together with a set-valued mapping $K:X\tos Y$, the conditions (A1)-(A3) (or (A1)-(A2$'$)), which are elements of the family $C(X\times Y)$ of all continuous bounded functions in $X\times Y$ equipped with the usual sup norm $\Vert u\Vert _{X\times Y,\infty}:=\sup\{\vert u(x,y)\vert: (x,y)\in X\times Y\}$, $u\in C(X\times Y)$. This question was thoroughly studied for the special case of the function $f\equiv 0$ in the paper of Kenderov and Lucchetti \cite{KL}. 

Having this new setting, the corresponding solution mapping 
$$
\tilde S_f: C(X\times Y)\tos X\times Y
$$ 
is defined in an obvious way: $\tilde S_f$ assigns to any $u\in C(X\times Y)$ the set (possibly empty) of the solutions to the supinf problem for the function $f+u$, with the set-valued mapping $K$.

It turns out  (with essentially the same proof) that this mapping share the same properties as the mapping $S_f$, namely we have:

\begin{prop}
Let $X$ and $Y$ be completely regular topological spaces,  and a given   extended real-valued function   $f:X\times Y\to [-\infty, +\infty]$ and a set-valued mapping $K:X\tos Y$ with nonempty images satisfy the conditions (A1)-(A2\,$'$). Suppose, in addition, that  $K$ is both lsc and usc in $X$ and with compact values. Then, the solution mapping $\tilde S_f:C(X\times Y)\tos X\times Y$  is with dense domain $Dom(\tilde S_f)$, closed graph $\Gr(\tilde S_f)$, quasi-continuous and open  as a mapping from $C(X\times Y)$ into $\Gr(K)$. 
\end{prop}

\begin{proof}\label{solution_map_2}
We only sketch the proof. The fact that $\Dom(\tilde S_f)$ is dense in $C(X\times Y)$ follows by the already mentioned fact that Lemma \ref{basic_lemma} holds also for any function of the type $f+u$, $u\in C(X\times Y)$, and the simple observation that for any $g\in C(X)$ and $h\in C(Y)$ we have
$$
\Vert g+h\Vert_{X\times Y,\infty} \le \Vert g\Vert_{X,\infty}+\Vert h\Vert_{Y,\infty}.
\leqno{(*)}
$$
Closedness of the graph if $\tilde S_f$ is proved following the path of the proof of the same fact as in Proposition \ref{solution_map_1} (we used there only general facts from uniform convergence of continuous bounded functions). The fact that $\tilde S_f$ is open as a mapping into $\Gr(K)$ is again proved as (d) above, having in mind (*). The proof of quasi-continuity of $\tilde S_f$ needs more careful consideration, but still it follows the same path as the proof of the same fact above, again having in mind (*).

\end{proof}

As it was mentioned above, the special case when $f\equiv 0$ was considered in \cite{KL} and similar properties of the solution mapping $\tilde S_f$ were obtained. The general case of a function $f$, which satisfies (A1)-(A2$'$), requires more attention.

\smallskip

We end this section by pointing out that when we have to deal with one level optimization, i.e. the objective function $f$ is only defined in the space $X$ and the perturbations are from $C(X)$, the corresponding solution mapping have the same properties as the solution mapping from Proposition \ref{solution_map_1} (see e.g. Proposition 2.4 from \cite{CKR3}). These properties can be derived from  Proposition  \ref{solution_map_1} (under its conditions) by supposing that the space $Y$ is a singleton. However, it should be noted that, in the case of one level optimization, in order to have the above properties for the corresponding solution mapping, the continuity of $f$ is needed only for condition (d) above. For the other properties (a)-(c) (as shown in Proposition 2.4 from \cite{CKR3}), the semicontinuity of the objective function  is enough.

\section{Generic variational principles for supinf problems}

In this section, for given completely regular topological spaces $X$ and $Y$,  an extended real-valued function $f:X\times Y\to [-\infty,\infty]$ and a  set-valued mapping $K:X\tos Y$ with nonempty images, we will be interested in under which conditions the set of the perturbations of $f$ from $C(X)\times C(Y)$ (or from $C(X\times Y)$) which secure existence of solutions to the perturbed supinf problem is not only dense (we have proved that under the conditions (A1)-(A3), the sets  $\Dom(S_f)$ and $\Dom(\tilde S_f)$ are dense in the corresponding spaces) but also residual in $C(X)\times C(Y)$ (or in  $C(X\times Y)$. The term {\it residual} for the set $\Dom(S_f)$ (or $\Dom(\tilde S_f)$), as usual,  means that $\Dom(S_f)$ (or $\Dom(\tilde S_f)$) contains a dense  and $G_\delta$-subset of $C(X)\times C(Y)$ (or of $C(X\times Y)$). From topological point of view, residual sets in complete metric spaces are considered to contain most of the elements from the space.

More precisely, we will be interested in conditions which ensure the above property not only for the of existence of solutions but also for  a stronger property  for the corresponding supinf problems called {\it well-posedness}. 

Let us first recall this notion in the case of optimization of a function of one variable. Given an  extended real-valued function $\varphi:Z\to\R\cup\{+\infty\}$ defined in a topological space $Z$ which is bounded from below, the problem to minimize $\varphi$ on $Z$   is said  to be  {\it Tykhonov well-posed} (or sometimes, simply, {\it well-posed}) if this problem has a unique minimum point  $z_0 \in Z$ and, moreover, for every minimizing sequence $\{z_n\}_{n=1}^{\infty} \subset  Z$ for this problem, i.e. $\varphi(z_n) \rightarrow \varphi(z_0)$, it follows that $z_n \rightarrow z_0$. An equivalent way to define the latter notion is: {\it every minimizing sequence for $\varphi$ converges to some minimizer of $\varphi$}. Symmetrically, we define this notion for maximization of a given bounded from above extended real-valued function \\ $\varphi:Z\to\R\cup\{-\infty\}$, by replacing minimum by maximum and minimization sequence by maximization sequence. 

This notion have significant importance both, in numerical optimization because it secures convergence of numerical algorithms for solving optimization problems, and also in variational and functional analysis (where it is known also under the name {\it the function $\varphi$  has a strong minimum (or maximum)}), where it is related, for example,  to the differentiability of certain classes of functions. The notion has been largely studied in many papers - see, e.g., the monographs \cite{DoZo} and \cite{DGoZ2} and the references therein. 

We will now introduce similar notions for sup-inf problems proposed in the paper \cite{KL}. Let be given an extended real-valued function $f:X\times Y\to [-\infty,\infty]$ which is proper and a  set-valued mapping $K\tos Y$ with nonempty images. We consider the constrained supinf problem for these data, namely,
$$
(P) \qquad \sup_{x \in X} \inf_{y \in Kx} f(x,y).
$$

In our setting, in order to introduce the notions proposed in \cite{KL}, we need to impose since the beginning, that the data of the problem satisfy condition (A1). Having done this, the first notion is related only to the first player (which somewhat reflects the role of this player in the leader-follower game): a point $x_0\in X$ is called  a solution for the leader player (called also {\it sup-solution}) if $v_f=\inf_{y\in Kx_0} f(x_0,y)$. The problem $(P)$ is said to be {\it sup-well-posed} if the optimization problem to maximize the function $w(\cdot):=\inf_{y\in K(\cdot)}f(\cdot, y)$ in $X$ is Tykhonov well-posed. In this case, obviously,  there is only one sup-solution.

A stronger notion of well-posedness relies on the following type  of optimizing sequences for $(P)$: A sequence $\{(x_n, y_n)\}_{n\ge 1} \in X \times Y$ is called {\it optimizing} for $(P)$ if:

\begin{enumerate}
  \item $y_n \in K x_n$ for every $n\ge 1$;
  \item $w(x_n)=\inf_{y\in Kx_n} f(x_n,y) \rightarrow v_f=\sup_{x \in X} \inf_{y \in Kx} f(x,y)$;
  \item $f(x_n,y_n) \rightarrow v_f$
\end{enumerate}
Condition (A1) ensures that the such optimizing sequences for the problem (P) exist.

The problem $(P)$ is said to be {\it well-posed } if  every optimizing sequence for the problem $(P)$ converges to some (actually, unique) solution $(x_0,y_0)$ of $P$. An equivalent notion was studied in \cite{LoM3} and a similar concept  was considered in \cite{Zo}. It is clear that if $(P)$ is well-posed with unique solution $(x_0,y_0)$, then the problem $(P)$ is sup-well posed with unique sup-solution $x_0$. The converse is not true, in general.

The following results, which were proved in \cite{GKaR}, relate the notion of well-posedness of supinf problems with continuity properties of the mappings $S_f$ and $\tilde S_f$:

\begin{theo} \label{equiv_separable} (\cite{GKaR}, Theorem  4.2)
Let the function $f:X\times Y\to [-\infty,+\infty]$ and the  set-valued mapping $K:X\tos Y$  satisfy the assumptions (A1)-(A3). Suppose that $K$ is lsc and with closed images. Then the supinf problem  (P) for the function $f+g+h$, with $(g,h)\in C(X)\times C(Y)$, is well-posed if, and only if, the  set-valued valued mapping $S_f:C(X)\times C(Y)\tos X\times Y$ is single-valued and usc at $(g,h)$. 

\end{theo}

\begin{theo} \label{equiv_general}  (\cite{GKaR}, Theorem  4.3)
Let the function $f:X\times Y\to [-\infty,+\infty]$ and the  set-valued mapping $K:X\tos Y$  satisfy the assumptions (A1)-(A3). Suppose that $K$ is lsc and with closed images. Then the supinf problem  (P) for the function $f+u$, with $u\in C(X\times Y)$, is well-posed if, and only if, the  set-valued  mapping $\tilde S_f:C(X\times Y)\tos X\times Y$ is single-valued and usc at $u$. 

\end{theo}

It should be noted, see,  e.g.,  \cite{CKR1},  that similar properties hold in the case of optimization of functions of one variable.

The next two results, which are generic variational principles for supinf problems, are consequences  of what was obtained above for the mappings $S_f$ and $\tilde S_f$, the last two theorems and general results for set-valued mappings:

\begin{theo}\label{var_principle_separable}
Let $X$ and $Y$ be completely regular topological spaces and suppose that the function $f:X\times Y\to [-\infty,+\infty]$ and the set-valued mapping $K:X\tos Y$  satisfy the assumptions (A1)-(A2\,$'$) and, in addition, the mapping $K$ is both usc and lsc with (nonempty) compact values. Suppose that the set $\Gr(K)$ (with the inherited from $X\times Y$ topology) contains a dense and completely metrizable subspace. Then the set $\{(g,f)\in C(X)\times C(Y):$ the supinf problem (P) for the function $f+g+h$ is well-posed\} contains a dense and $G_\delta$-subset of $C(X)\times C(Y)$. 
\end{theo}

\begin{theo}\label{var_principle_general}
Let $X$ and $Y$ be completely regular topological spaces and suppose that the function $f:X\times Y\to [-\infty,+\infty]$ and the  set-valued mapping $K:X\tos Y$  satisfy the assumptions (A1)-(A2\,$'$)  and, in addition, the mapping $K$ is both usc and lsc with (nonempty) compact values. Suppose that the set $\Gr(K)$ (with the inherited from $X\times Y$ topology) contains a dense and completely metrizable subspace. Then the set $\{u\in C(X\times Y):$ the supinf problem  (P) for the function $f+u$ is well-posed\} contains a dense and $G_\delta$-subset of $C(X\times Y)$. 
\end{theo}

\begin{proof}

Both mappings, $S_f$ and $\tilde S_f$,  are with closed graphs, dense domains, open and quasi-continuous in the respective definition spaces. From general results for set-valued mappings, see e.g. \cite{CKR2}, Theorem 5.3, it follows that  both mappings are single-valued and usc at the points of a dense and $G_\delta$-subset of their corresponding definition spaces. The rest follows from the Theorems \ref{equiv_separable} and \ref{equiv_general}.

\end{proof}

\begin{rem}{\rm
The last theorem was proved in \cite{KL} for the special case $f\equiv 0$. In fact, in the same paper it was shown that the  condition that $\Gr(K)$ contains a dense completely metrizable subspace is equivalent to the generic variational principle in $C(X\times Y)$ for the function $f\equiv 0$. This means that the validity of the generic variational principle  for any function $f$ as in the theorem, is also equivalent to the fact that the graph $\Gr(K)$ contains a dense completely metrizable subspace.
}
\end{rem} 

We have two important partial cases of the above theorems. The first one is immediate, having in mind that the graph of any usc and compact-valued set-valued mapping is always closed.

\begin{cor}
Let $X$ and  $Y$  be complete metric spaces and suppose that the function $f:X\times Y\to [-\infty,+\infty]$ and the set-valued mapping $K:X\tos Y$  satisfy the assumptions (A1)-(A2\,$'$)  and, in addition, the mapping $K$ is both usc and lsc with (nonempty) compact values.  Then,

\begin{enumerate}

\item[{\rm (a)}] the set $\{(g,h)\in C(X)\times C(Y):$ the  supinf problem (P) for the function  $f+g+h$ is well-posed\} contains a dense and $G_\delta$-subset of $C(X)\times C(Y)$; 

\item[{\rm (b)}] the set $\{u\in C(X\times Y):$ the  supinf problem (P) for the function $f+u$ is well-posed\} contains a dense and $G_\delta$-subset of $C(X\times Y)$. 

\end{enumerate} 

\end{cor}

 The next is a generic variational principle for unconstrained supinf problems.

\begin{theo}\label{var_principle_compact}
Let $X$ and $Y$ be completely regular topological spaces each of which contains a dense completely metrizable subspace. Suppose that $Y$ is compact and that the function $f:X\times Y\to [-\infty,+\infty]$ satisfy the assumptions (A1)-(A2\,$'$)  with respect to the  mapping $Kx:=Y$,  $x\in X$. Then,
\begin{enumerate}

\item[{\rm (a)}] the set $\{(g,h)\in C(X)\times C(Y):$ the unconstrained supinf problem (P) for the function $f+ g+h$ is well-posed\} contains a dense and $G_\delta$-subset of $C(X)\times C(Y)$; 

\item[{\rm (b)}] the set $\{u\in C(X\times Y):$ the unconstrained supinf problem (P) for the function $f+u$ is well-posed\} contains a dense and $G_\delta$-subset of $C(X\times Y)$. 

\end{enumerate} 

\end{theo}

\begin{proof}

The mapping $K:X\tos Y$ given by $Kx=Y$, $x\in X$, is evidently both usc and lsc and with compact values and the graph $\Gr(K)$ is $X\times Y$. It remains to apply Theorems \ref{var_principle_separable} and \ref{var_principle_general}.
\end{proof}

It is possible (as it is for the case for one level optimization) to obtain  generic variational principles for supinf problems  for a notion of well-posedness which generalizes the well-posedness introduced above, in the spirit  that it does not require that the solution is necessarily unique. Namely, following the idea for the same notion in one level optimization, we say  that the problem (P) for the function $f:X\times Y\to [-\infty,+\infty]$  and the set-valued mapping $K:X\tos Y$, which satisfy (A1),  is {\it well-posed in generalized sense}, if every optimizing {\bf net} has  a convergent to a solution subnet. In particular, in such a case the set of solutions of (P) is nonempty and compact. It should be mentioned (and this is easily seen, as in the case of one level optimization) that in the case of the usual well-posedness with unique solution, we can confine ourselves to work only with optimizing sequences, while in this generalized case, we must consider nets, since there are no first countable conditions imposed on the underlying spaces. It can be seen, with essentially the same proofs as in Theorems \ref{equiv_separable} and \ref{equiv_general}, that the mapping $S_f$ (resp. $\tilde S_f$) is usc and (nonempty) compact-valued at $(g,h)\in C(X)\times C(Y)$ (resp. at $u\in C(X\times Y)$) if, and only if the supinf problem for the function $f+g+h$ (resp. for the function $f+u$) is generalized well-posed. Having this in hand, the following theorem is proved exactly as Theorems \ref{var_principle_separable} and \ref{var_principle_general} using, instead of Theorem 5.3 from \cite{CKR2}, Theorem 5.2 from the same paper.

Recall that a completely regular topological space is called {\it \v Cech complete} if it lies as a $G_\delta$-subset in some of its compactifications. 

\begin{theo}
Let $X$ and $Y$ be completely regular topological spaces and suppose that the function $f:X\times Y\to [-\infty,+\infty]$ and the  set-valued mapping $K:X\tos Y$  satisfy the assumptions (A1)-(A2\,$'$)  and, in addition, the mapping $K$ is both usc and lsc with (nonempty) compact values.  Suppose that $\Gr(K)$ contains a dense \v Cech complete subspace. Then,

\begin{enumerate}

\item[{\rm (a)}] the set $\{(g,h)\in C(X)\times C(Y):$ the  supinf problem (P)  for the function $f+ g+h$ is generalized well-posed\} contains a dense and $G_\delta$-subset of $C(X)\times C(Y)$; 

\item[{\rm (b)}] the set $\{u\in C(X\times Y):$ the  supinf problem (P) for the function $f+u$ is generalized well-posed\} contains a dense and $G_\delta$-subset of $C(X\times Y)$. 

\end{enumerate} 

\end{theo}

\medskip

We will end this section by considering the special case of Theorem \ref{var_principle_compact}, when $Y$ is a singleton. In such a setting the space $C(X\times Y)$ can be identified with the space $C(X)$. Moreover, if we perturb a fixed function $f:X\times Y\to [-\infty,+\infty]$, the supinf problem for this function and for all its perturbations by functions from $C(X\times Y)$ become only sup problems for the function $f$ on the first variable. And the well-posedness of the corresponding supinf problem, is in fact, well-posedness in the usual Tykhonov sense of the sup problem for the perturbed function.  Observe also that condition (A1) excludes the situation in which the function $f$ (and its perturbations) take value $+\infty$. Therefore, we have the following corollary from the Theorem \ref{var_principle_compact}:

\begin{cor} 
Let $X$ be a completely regular topological space and $f:X\to[-\infty,+\infty)$ be a proper  extended real-valued function which is continuous and bounded above. If $X$ contains a dense completely metrizable subspace, then the set $\{g\in C(X):$ the maximization problem for the function  $ f+g$ is Tykhonov well-posed\} contains a dense and  $G_\delta$-subset of $C(X)$.

\end{cor}

\begin{proof}
The proof follows by the remarks before the formulation and the obvious fact that the graph $\Gr(K)$ in this case can be identified with $X$.
\end{proof}

This result was obtained in \cite{CKR1} for the case $f\equiv 0$ (by the way, from the latter, it follows that the same result is true for any real-valued continuous bounded function $f$ in $X$, since the translations in $C(X)$ are homeomorphisms). In fact, in the same paper, it was shown that the validity of the generic variational principle with well-posedness  for $f\equiv 0$ is equivalent to the property that the underlying space $X$ contains a dense completely metrizable subspace. We see from the last corollary that the result is true for a larger class of objective functions $f$. Moreover, when the function $f$  is real-valued (i.e. everywhere finite valued) we do not necessary need that $f$ is continuous  -  as it was shown in  the recent paper \cite{IvKR}, to have a generic variational principle in this case,  as in the last corollary, only lower semicontinuity (even a weaker notion) of the objective function $f$  is enough.

\bigskip

D. Gaumont, CRED, University of Pantheon-Assas, 31 Rue Froidevaux, 75014 Paris, France, \\ e-mail: damien.gaumont@u-paris2.fr

D. Kamburova, Faculty of Mathematics and Informatics, Sofia University, 5, James Bourchier blvd., 1164 Sofia, Bulgaria \\
and \\
Institute of Mathematics and Informatics, Bulgarian Academy of Sciences, Acad. G. Bonchev str., block 8, 1113 Sofia, Bulgaria; \\ e-mail: detelinak@math.bas.bg

J.P. Revalski, International Center for Mathematical Sciences, Institute of Mathematics and Informatics, Bulgarian Academy of Sciences, Acad. G. Bonchev str., block 8, 1113 Sofia, Bulgaria \\
and \\
Centre of Excellence in Informatics and Information and Communication Technologies, Sofia, Bulgaria.
\\ e-mail: revalski@math.bas.bg


\begin{thebibliography}{XX}



\bibitem{CKR1} M.M. \v Coban, P.S. Kenderov and J.P. Revalski, Generic well-posedness  of  optimization  problems   in   topological   spaces, {\it Mathematika} {\bf 36} (1989), 301--324.

\bibitem{CKR2} M.M. \v Coban, P.S. Kenderov and J.P. Revalski, Densely defined selections of multivalued mappings, {\it Trans. Amer.  Math. Soc.}  {\bf 344} (1994), 533--552.

\bibitem{CKR3} M.M. Choban, P.S. Kenderov and J.P. Revalski, Variational principles and topological games, {\it Topology and its Appl.}, {\bf 159} (2012), 3550--3562.

\bibitem{DGoZ1} R. Deville, G. Godefroy and  V. Zizler, A smooth variational principle with applications to Hamilton--Jacobi equations in infinite dimensions,  {\it J. Funct. Anal.} {\bf 111} (1993), 197--212.

\bibitem{DGoZ2} R. Deville, G. Godefroy and  V. Zizler, Smoothness and renormings in Banach spaces,  Pitman Monogr. Surveys Pure Appl. Math., {\bf 64}, Longman Scientific \& Technical, Harlow, 1993.

\bibitem{DoZo} A. Dontchev, T. Zolezzi,  Well-posed optimization problems, {\it Lecture Notes in Math.} {\bf 1543},  Springer-Verlag, Berlin,
1993.

\bibitem{Chr} J.P.R. Christensen, Theorems of  I. Namioka  and R.E. Johnson type for upper  semi-continuous and compact-valued set-valued maps, {\it Proc. Amer. Math. Soc.} {\bf 86} (1982), 649--655.


\bibitem{GKaR} D. Gaumont, D. Kamburova and J.P. Revalski, Perturbations of supinf problems with constraints, {\it Vietnam J. Math.},  {\bf 47}, No.3, (2019),  659--667.

\bibitem{ILR} A. Ioffe,  R. Lucchetti and  J.P. Revalski,   A variational principle for problems with functional constraints, {\it SIAM J. Optimization}, {\bf 12}, No.2, (2001), 461--478 . 

\bibitem {IZa} A.D. Ioffe, A.J. Zaslavski, Variational principles and well-posedness in optimization and calculus of variations, {\it SIAM J. Control Optimization} {\bf 38} (2000), 566--581.

\bibitem{IvKR} M. Ivanov, P.S. Kenderov and J.P. Revalski, Variational principles for maximization problems with lower-semicontinuous goal functions, {\it Set-valued and Variational Analysis},  {\bf 30} (2022),  559–-571.   https://doi.org/10.1007/s11228-021-00604-1
    
\bibitem{Ke} S. Kempisty, Sur les fonctions quasi-continues, {\it Fund. Math.}, {\bf 19} (1932), 184--197.    

\bibitem{KL} P.S. Kenderov  and R. Lucchetti,  Generic well-posedness of supinf problems, {\it   Bull. Austral. Math. Soc.} {\bf 54} (1996), 5--25.


\bibitem{KR2} P.S. Kenderov P. and  J.P. Revalski,  Variational principles for supinf problems,  {\it Compt. rend. Acad. bulg. Sci.} 70(12) (2017), 1635--1642 .



\bibitem{KRi} P.S. Kenderov and N. Ribarska, Generic uniqueness of the solution of "min-max" problems, in {\it Lecture Notes in Mathematical Systems}, {\bf 304} (1988), pp. 41--48, Springer. 

\bibitem{LaR} M. Lassonde and J.P. Revalski, Fragmentability of sequences of set-valued mappings with application to variational principles, {\it  Proc. Amer. Math. Soc.}, {\bf  133} (2005), 2637–-2646. 

\bibitem{LoM1} P. Loridan  and  J. Morgan, New results on approximate solution in two-level optimization, {\it Optimization}, {\bf 20}(6) (1989), 819--836.


\bibitem{LoM3} P. Loridan  and  J. Morgan,  $\eps$-regularized two level optimization problems; approximation and existing results, in {\it Lecture Notes in Mathematics}, {\bf 1405} (1989), pp. 99--113, Springer.




\bibitem{TZl2} H. Topalova, N. Zlateva, Generic continuity of the perturbed minima of certain parametric optimization problems, {\it Positivity}, {\bf 29}(3) (2025), 1--17.

\bibitem{H_Stack} H. von Stackelberg, Market Structure and Equilibrium, Springer, (2011).

\bibitem{S} C. Stegall,  Optimization of functions on certain  subsets of Banach spaces, {\it Math. Ann.} {\bf 236} (1978), 171--176.

\bibitem{Zo} T. Zolezzi, Well-posed control problems: a perturbation approach, in B.S.  Mordukhovich et all (eds.), Nonsmooth Analysis and Geometric Methods in Deterministic Optimal Control, IMA Proceedings vol. 78, Springer (1996), pp. 239--246.


\end{thebibliography}
\end{document}